\documentclass[12pt]{amsart}
\usepackage{amsmath, amsthm, amssymb}

\newtheorem{theorem}{Theorem}[section]
\newtheorem{lemma}[theorem]{Lemma}
\newtheorem{proposition}[theorem]{Proposition}
\newtheorem{ex}{Exercise}[section]
\newtheorem{examp}[theorem]{Example}

\newtheorem{remar}[theorem]{Remark}
\newcommand{\n}{\par\noindent}

\newcommand{\sn}{\par\smallskip\noindent}

\newcommand{\pars}{\par\smallskip}
\newcommand{\parm}{\par\medskip}

\newcommand{\chara}{\mbox{\rm char}\,}
\newcommand{\bfind}[1]{\index{#1}{\bf #1}}
\newcommand{\sep}{^{\rm sep}}
\newcommand{\Gal}{\mbox{\rm Gal}\,}

\newcommand{\R}{\mathbb R}

\newcommand{\cO}{\mathcal{O}}

\begin{document}
\title[Subfields of Kaplansky fields]{Subfields of algebraically maximal Kaplansky fields}
\thanks{The author would like to thank Anna Blaszczok for the careful proofreading of an earlier version
of this manuscript, and Pablo Cubides Kovacsics for helpful remarks.\\
The initial work for this paper was done while the author was a visiting professor at the Institute of 
Mathematics at the University of Szczecin; he would like to thank his colleagues at this Institute for their 
great hospitality and support.}
\author{F.-V.~Kuhlmann}
\address{Institute of Mathematics, University of Szczecin, ul. Wielkopolska 15, 	  	  	
70-451 Szczecin, Poland}
\email{fvk@math.us.edu.pl}
\date{21.\ 12.\ 2017}

\subjclass[2000]{Primary 12J10; Secondary 12J25, 13A18, 12Y05}
\keywords{valued field, Kaplansky field, tame field, Newton algorithm}

\begin{abstract}
Using the ramification theory of tame and Kaplansky fields, we show that maximal Kaplansky fields contain
maximal immediate extensions of each of their subfields. Likewise, algebraically maximal Kaplansky fields
contain maximal immediate algebraic extensions of each of their subfields. This study is inspired by
problems that appear in henselian valued fields of rank higher than 1 when a Hensel root of a polynomial
is approximated by the elements generated by a (transfinite) Newton algorithm.
\end{abstract}

\maketitle

%
%
\section{Introduction}
For real functions, the Newton Algorithm is a nice tool to approximate their zeros. An analogue works in
valued fields, such as the fields of $p$-adic numbers. Take a complete discretely valued field $K$ with
valuation ring $\cO$ and a polynomial $f$ in one variable with coefficients in $\cO$. If $b\in\cO$
satisfies $vf(b)>2vf'(b)$, then set $x_0:=b$ and
\[
x_{i+1}\>:=\>x_i\,-\,\frac{f(x_i)}{f'(x_i)}
\]
for $i\geq 0$. It can easily be shown, using the Taylor expansion of $f$, that this sequence is a Cauchy
sequence, and if a limit exists, then it is a root of $f$. This fact can be used to prove Hensel's Lemma
in the complete fields of $p$-adic numbers, which states the existence of $p$-adic roots
of suitable polynomials.

Hensel's Lemma also holds in other valued fields (such as power series fields
with arbitrary ordered abelian groups of exponents), which can be much larger than the fields of $p$-adic
numbers. Still, it can be proved using the Newton algorithm, but if the value group of the field is not
archimedean ordered (i.e., if it is not an ordered subgroup of $\R$), then the
algorithm may not deliver the root after the first $\omega$ many iterations and transfinite induction is
needed. For those who do not like the technicalities of transfinite induction, S.~Prie{\ss}-Crampe
presented in 1990 an elegant alternative (see \cite{[P]}). The setting of the Newton algorithm gives rise
to a contracting function
\[
\Phi(x)\>:=\> x\,-\,\frac{f(x)}{f'(x)}
\]
defined on a suitable subset of the field, whose unique fixed point, if it exists, is a root of the
polynomial. The existence is guaranteed if $f$ satisfies the assumptions of Hensel's Lemma and the
valued field from which its coefficients are taken is spherically complete (which is a property shared
by all power series fields). While it is not the original definition, it suffices here to say that a valued
field is spherically complete if and only if every pseudo Cauchy sequence in the sense of \cite{[Ka]}
has a (pseudo) limit.

The members of the sequence obtained from the Newton algorithm are approximations to a root of $f$. If the
value group of the field in which we are working is not archimedean ordered, then the problem can appear
that the sequence is not Cauchy, but only a pseudo Cauchy sequence. The limits of pseudo Cauchy
sequences which are not Cauchy are not uniquely determined.
In a recent paper \cite{[P2]}, Prie{\ss}-Crampe considers the question whether this situation can be
repaired, for a given polynomial $f$, by passing to a suitable subfield of $K$ in which the sequence
becomes Cauchy and has a limit. She proves that this can indeed be done when $K$ is a spherically complete
Kaplansky field. Let us provide the necessary definitions and background.

\pars
We will work with (Krull) valuations $v$ and write them in the classical additive way, that
is, the value group of $v$ on a field $K$, denoted by $vK$, is an additively written ordered abelian group,
and the ultrametric triangle law reads as $v(a+b)\geq\min \{va,vb\}$. We denote by $Kv$ the residue field
of $v$ on $K$, by $va$ the value of an element $a\in K$, and by $av$ its residue.

A polynomial $f$ over a field of characteristic $p>0$ is called \bfind{additive} if $f(a+b)=f(a)+f(b)$
for all elements $a,b$ in any extension of the field. This holds if and only if $f$ is of the form
$\sum_{0\leq i\leq n} c_i X^{p^i}$. Following I.~Kaplansky, we call a polynomial $g$ a \bfind{$p$-polynomial}
if $g=f+c$ where $f$ is an additive polynomial and $c$ is a constant.

A valued field $(K,v)$ is called a \bfind{Kaplansky field} if $\chara Kv=0$ or if it satisfies
\bfind{Kaplansky's hypothesis A}: for $\chara Kv=p>0$,
\sn
{\bf (A1)} \ every $p$-polynomial with coefficients in $Kv$ has a zero in $Kv$,
\sn
{\bf (A2)} \ $vK$ is $p$-divisible.

\pars
Prie{\ss}-Crampe proves her main result in \cite{[P2]} by a direct application of this formulation of
hypothesis A. While this proof is not without interest, a more modern approach can lead to more insight,
as we will show in the present paper. A lot of important work has been done over time by several authors
in order to better understand Kaplansky's hypothesis A, e.g.\ by G.~Whaples in \cite{[Wh]}, by F.~Delon
in her thesis, and in the paper \cite{[KPR]}. The final touch to this development was given by Kaplansky
himself, in cooperation with D.~Leep (see \cite[Section 9]{[Ku1]}). We will use the following
characterization. By Theorem 1 of \cite{[Wh]} and the other cited sources,
hypothesis A is equivalent to the conjunction of the following three conditions,
where $p$ denotes the characteristic of the residue field:
\sn
{\bf (K1)} if $p>0$, then the value group is $p$-divisible,
\sn
{\bf (K2)} the residue field is perfect,
\sn
{\bf (K3)} the residue field admits no finite separable extension of degree divisible by $p$.

\pars
A valued field is \bfind{henselian} if it satisfies Hensel's Lemma, or equi\-valently, admits a
unique extension of its valuation to its algebraic closure. Note that every algebraic extension of a
henselian field is again henselian, with respect to the unique extension of the valuation. A
\bfind{henselization} of a valued field $(K,v)$ is an algebraic extension which is henselian and minimal
in the sense that it can be embedded over $K$ in every other henselian extension field of $(K,v)$.
Henselizations exist for every valued field $(K,v)$, in fact, every henselian extension field of $(K,v)$
contains a henselization of $(K,v)$. Henselizations are unique up to valuation preserving isomorphism over
$K$. Therefore, we will speak of ``the henselization of $(K,v)$'' and denote it by $(K^h,v)$.

An extension $(L|K,v)$ of valued fields is called \bfind{immediate} if the canonical embeddings of $vK$ in
$vL$ and of $Kv$ in $Lv$ are onto, in other words, value group and residue field remain unchanged. A valued
field is called \bfind{maximal} if it does not admit any nontrivial immediate extensions; by
\cite[Theorem 4]{[Ka]}, this
holds if and only if it is spherically complete. A valued field is called \bfind{algebraically
maximal} if it does not admit any nontrivial immediate algebraic extension. Since henselizations are
immediate algebraic extensions, every algebraically maximal field is henselian.

The notion of ``purely wild extension'' (an algebraic extension of a henselian field that is linearly
disjoint from all tame extensions) was introduced in \cite{[KPR]}. We will give the precise definitions
for both types of extensions in the Preliminaries Section.

\pars
The following is our main theorem:

\begin{theorem}                               \label{MT}
Take a valued field extension $(L|K,v)$, with $(L,v)$ an algebraically maximal Kaplansky field. Then $L$
contains a maximal immediate algebraic extension of $(K,v)$, as well as a maximal purely wild extension
of the henselization of $(K,v)$ inside of $(L,v)$.

If $(L,v)$ is maximal, then it also contains a maximal immediate extension of $(K,v)$.
\end{theorem}

Modulo the fact that a valued field is maximal if and only if it is spherically complete, the last
assertion of this theorem is proved in \cite[Theorem 3.5]{[P2]}. As already indicated, we present a
different proof, which will be based on the ramification theory of Kaplansky and tame fields.
A crucial tool in this proof is the following analogue of Lemma~3.7 of \cite{[Ku2]}, which deals with
the case of tame fields. Interestingly, in the case of algebraically maximal Kaplansky fields
we do not need the assumption of that lemma that the residue field extension $Lv|Kv$ be algebraic.

\begin{proposition}                         \label{p2}
Take an algebraically maximal Kaplansky field $(L,v)$ and let $K$ be a relatively algebraically closed
subfield of $L$. Then also $(K,v)$ is an algebraically maximal Kaplansky field, with its residue field
relatively algebraically closed in that of $L$.
\end{proposition}

\parm
These results are a nice complement to the theory of tame fields as developed in \cite{[Ku2]}. The proofs
of the above results will be given in Section~\ref{sectres}, along with some more facts about algebraically
maximal Kaplansky fields.

%
%
\section{Preliminaries}
We start with the following well known fact:

\begin{lemma}                                        \label{hrach}
If $(L,v)$ is henselian and $K$ a relatively algebraically closed subfield of $L$, then also $(K,v)$ is
henselian.
\end{lemma}
\begin{proof}
Assume that $f$ is a polynomial with coefficients in the valuation ring of $(K,v)$ which satisfies the
conditions of Hensel's Lemma. Since the valuation ring of $(K,v)$ is contained in that of $(L,v)$ and
$(L,v)$ is henselian, $f$ admits a root $a$ in $L$ which satisfies the assertions of Hensel's Lemma. Being
a root of $f$, $a$ is algebraic over $K$ and since $K$ is relatively algebraically closed in $L$, we have
that $a\in K$.
\end{proof}

Take a finite extension $(L|K,v)$ of valued fields. The \bfind{Lemma of Ostrowski} says that whenever
the extension of $v$ from $K$ to $L$ is unique, then
\begin{equation}                            
[L:K]\;=\; p^\nu \cdot (vL:vK)\cdot [Lv:Kv] \;\;\;\mbox{ with }
\nu\geq 0\>,
\end{equation}
where $p$ is the \textbf{characteristic exponent} of $Lv$, that is,
$p=\chara Lv$ if this is positive, and $p=1$ otherwise. For the proof,
see \cite[Theoreme 2, p.~236]{[R]}) or \cite[Corollary to Theorem~25, p.~78]{[ZS]}).
If $p^\nu=1$, then we say that the extension $(L|K,v)$ is \bfind{defectless}.
Note that $(L|K,v)$ is always defectless if $\chara Kv=0$. We call a henselian field a \bfind{defectless
field} if all of its finite extensions are defectless. Each valued field of residue characteristic $0$ is a
defectless field.

An algebraic extension $(L|K,v)$ of henselian fields is called \bfind{tame} if every finite subextension
$E|K$ of $L|K$ satisfies the following conditions:
\sn
{\bf (TE1)} if $\chara Kv=p>0$, then the ramification index $(vE:vK)$ is prime to $p$,\sn
{\bf (TE2)} the residue field extension $Ev|Kv$ is separable,\sn
{\bf (TE3)} the extension $(E|K,v)$ is defectless.

\pars
A \bfind{tame valued field} (in short, \bfind{tame field}) is a
henselian field for which all algebraic extensions are tame. From the definition of a tame extension it
follows that $(K,v)$ is a tame field if and only if it is a defectless field satisfying conditions (K1) and
(K2). It is an easy observation that every extension of a defectless field $(K,v)$ satisfying conditions (K1)
and (K2) must be tame and so $(K,v)$ must be a tame field. The converse can be derived from the 
arguments in the second part of the proof of Proposition~\ref{p1} below.

By (K1) and (K2), the perfect hull of a tame field is an immediate extension, and by (TF3), this extension
must be trivial. This shows that every tame field is perfect.

\pars
If $\chara Kv=0$, then conditions (TE1) and (TE2) are void, and
every finite extension of $(K,v)$ is defectless. Hence very algebraic extension of a henselian field of
residue characteristic 0 is a tame extension, and every henselian field of residue characteristic 0
is a tame field.

\begin{lemma}                           \label{amKf=tame}
Every algebraically maximal valued field satisfying conditions (K1) and (K2), and in particular every
algebraically maximal Kaplansky field, is a tame field and hence a perfect and defectless field.
\end{lemma}
\begin{proof}
This follows from \cite[Theorem~3.2]{[Ku2]} and the facts about tame fields we have mentioned above.
\end{proof}

The converse of this lemma does not hold since for a tame field it is admissible that its residue
field has finite separable extensions with degree divisible by~$p$.

\parm
Take a valued field $(K,v)$, fix an extension of $v$ to $K\sep$ and call it again $v$. The fixed field of
the closed subgroup
\[
G^r:=\{\sigma \in \Gal(K\sep|K) \mid v(\sigma a-a)>va \textrm{ for all }
 a\in \mathcal{O}_{K\sep}\setminus\{0\}\}
\]
of Gal$(K\sep|K)$ (cf.~\cite[Corollary (20.6)]{[E]}) is called the \bfind{absolute ramification field of
$(K,v)$}.
For the following fact, see \cite[(20.15 b)]{[E]}.

\begin{lemma}                         \label{Z}
Take a henselian field $(K,v)$ and denote by $Z$ the absolute ramification field of $(K,v)$. If $(K',v)$ is
an extension of $(K,v)$ inside of $Z$, then $Z$ is also the absolute ramification field of $(K',v)$.
\end{lemma}

The next result follows from \cite[Theorem (22.7)]{[E]} (see also \cite[Proposition~4.1]{[KPR]}).

\begin{proposition}                         \label{Zmaxte}
The absolute ramification field of $(K,v)$ is the maximal tame extension of $(K,v)$.
\end{proposition}

An algebraic extension of a henselian field is called \bfind{purely wild} if it is linearly disjoint to the
absolute ramification field and thus to every tame extension. In \cite{[KPR]} it was shown that maximal
purely wild extensions are field complements to the absolute ramification field. In the case of
Kaplansky fields, they are at the same time maximal immediate algebraic extensions, and they are unique up
to isomorphism. See the cited paper for details.

\pars
Finally, we will need the following results about the absolute ramification field.

\begin{lemma}                         \label{Zcap}
Take an algebraic extension $(L|K,v)$ of henselian fields. Denote by $Z$ the absolute ramification field
of $(K,v)$. Then $(Z\cap L,v)$ is a tame extension of $(K,v)$ maximal with respect to being
contained in $L$, and $(L,v)$ is a purely wild extension of $(Z\cap L,v)$.
\end{lemma}
\begin{proof}
By Proposition~\ref{Zmaxte}, $(Z\cap L,v)$ is a tame extension of $(K,v)$.
Assume that $(E,v)$ is a tame extension of $(K,v)$ inside of $(L,v)$ which is an
extension of $(Z\cap L,v)$. Then $(E,v)$ is also a tame extension of $(Z\cap L,v)$.
By Lemma~\ref{Z}, $Z$ is also the absolute ramification field of $(Z\cap L,v)$. Thus by
Proposition~\ref{Zmaxte}, $Z$ is the maximal tame extension of $(Z\cap L,v)$. Hence $E\subseteq Z$ and
therefore, $E=Z\cap L$, which proves that $(Z\cap L,v)$ is a tame extension of $(K,v)$ maximal
with respect to being contained in $L$.

Now we wish to prove that $(L|Z\cap L,v)$ is a purely wild extension. Since $Z$ is the absolute
ramification field of $(Z\cap L,v)$, it suffices to show that $L|Z\cap L$ is linearly disjoint from $Z|Z
\cap L$. But this follows from the fact that according to \cite[Theorem~5.3.3 (2)]{[EP]}, $Z|K$ is a
Galois extension.
\end{proof}

%
%
\section{Results on algebraically maximal Kaplansky fields}             \label{sectres}
%
%
We start with the following characterization of algebraically maximal Kaplansky fields:

\begin{proposition}    \label{p1}
A valued field $(K,v)$ is an algebraically maximal Kaplansky field if and only
if it is henselian and does not admit any finite extension of degree divisible by $p$.
\end{proposition}
\begin{proof}
If $\chara Kv=0$, then the assertion is trivial since in this case, $(K,v)$ is an algebraically maximal
Kaplansky field if and only if it is henselian, and $0$ does not divide the degree of any finite extension.

Now we consider the case of $\chara Kv=p>0$. Assume first that $(K,v)$ is an algebraically maximal Kaplansky
field. Then in particular, it is henselian. By Lemma~\ref{amKf=tame}, it is a defectless field, that is,
every finite extension $(K'|K,v)$ satisfies
\[
[K':K]\>=\> (vK':vK)\cdot [K'v:Kv]\>.
\]
Since $(K,v)$ satisfies (K1), $(vK':vK)$ is not divisible by $p$. Because of (K2) and (K3), the same holds
for $[K'v:Kv]$. Hence also $[K':K]$ is not divisible by $p$.

Now assume that $(K,v)$ is henselian and does not admit any finite extension of degree divisible by $p$. By the Lemma of Ostrowski, every finite extension $(K'|K,v)$ satisfies
\[
[K':K]\>=\> p^{\nu}\cdot (vK':vK)\cdot [K'v:Kv]
\]
with $\nu\geq 0$. Since $[K':K]$ is not divisible by $p$, it follows that $\nu=0$. This implies that
$(K,v)$ is a defectless field and thus also an algebraically maximal field. Further, if there were an
element $\alpha\in vK$ not divisible by $p$, then adjoining to $K$ the $p$-th root of any element of $K$
having value $\alpha$ would generate an extension of $K$ of degree $p$. This shows that $vK$ must be
$p$-divisible, i.e., $(K,v)$ satisfies (K1). Finally, if there were a finite extension of $Kv$ with a
degree divisible by $p$, then it could be lifted to an extension of $K$ of the same degree, which should
not exist. This shows that $Kv$ does not admit any finite extension of degree divisible by $p$, so $(K,v)$
satisfies (K2) and (K3).
\end{proof}

\sn
{\it Proof of Proposition \ref{p2}.}
The first assertion is trivial if $\chara Kv=0$ since a relatively algebraically closed subfield of a
henselian field is again henselian. Suppose that $Kv$ is not relatively algebraically closed in $Lv$.
Then there exists some $\bar a\in Lv\setminus Kv$ separable-algebraic over $Kv$. Lift
its minimal polynomial $\bar f$ over $Kv$ up to a polynomial $f$ over $K$ of the same degree. Since
$\bar a$ is a simple root of $\bar f$, Hensel's Lemma yields the existence of a root $a\in L$ of $f$
with $av=\bar a$. Then $a$ is algebraic over, but not in $K$, which contradicts our assumption
that $K$ is relatively algebraically closed in $L$. Hence $Kv$ is relatively algebraically closed in $Lv$.

Let us now assume that $\chara Kv=p>0$. Then we can use the assertion of Proposition~\ref{p1}. Since $L$ is
perfect by Lemma~\ref{amKf=tame}, the same holds for its relatively algebraically closed subfield $K$.
Since the algebraically maximal field $(L,v)$ is henselian, the same holds for $(K,v)$ by Lemma~\ref{hrach}.

Take any finite extension $K'|K$. Since $K$ is perfect and relatively algebraically closed in $L$,
$K'|K$ is linearly disjoint from $L|K$. Therefore, $[K':K]=[L.K':L]$, which is not divisible by $p$. In view
of Proposition~\ref{p1}, this proves that $(K,v)$ is an algebraically maximal Kaplansky field. In particular,
$Kv$ is perfect. As in the case of $\chara Kv=0$ it is now shown that $Kv$ is relatively algebraically closed
in $Lv$.
\qed

\parm
%

\begin{proposition}              \label{pZ}
Take an algebraically maximal Kaplansky field $(L,v)$ and a henselian subfield $K$ such that $L|K$ is
algebraic. Denote the absolute ramification field of $(K,v)$ by $Z$ and set $Z_0:=Z\cap L$. Then the
following assertions hold:
\sn
a) \ $(Z_0,v)$ is a tame extension of $(K,v)$ maximal with respect to being contained in $L$,
\n
b) \ $(L,v)$ is a maximal purely wild extension of $(Z_0,v)$,
\n
c) \ {\rm (K3)} also holds for $(Z_0,v)$.
\end{proposition}
\begin{proof}
The assertions are trivial if $\chara Kv=0$ (in which case $Z_0=L$), so let us assume again that
$\chara Kv=p>0$.

It follows from Lemma~\ref{Zcap} that $(Z_0,v)$ is a tame extension of $(K,v)$ maximal with respect to
being contained in $L$, and that $(L|Z_0,v)$ is purely wild. In fact, $(L,v)$ is a maximal purely wild
extension of $(Z_0,v)$ since by Lemma~\ref{amKf=tame} it is a tame field.
Since $(L|Z_0,v)$ is purely wild, the extension of the respective residue fields is purely inseparable.
As $L$ satisfies (K3), the same is consequently true for $Z_0\,$.
\end{proof}

\begin{proposition}                                    \label{prop3}
Take an algebraically maximal Kaplansky field $(L,v)$ and a subfield $K$ such that $L|K$ is algebraic. Then
$(L,v)$ contains:
\sn
a) \ a maximal immediate algebraic extension of $(K,v)$,
\n
b) \ a henselization $(K^h,v)$ of $(K,v)$ and a maximal purely wild extension of $(K^h,v)$.
\end{proposition}
\begin{proof}
As before, the assertions are trivial if $\chara Kv=0$, so let us assume
again that $\chara Kv=p>0$. Since the algebraically maximal field $(L,v)$ is henselian, it
contains a henselization $K^h$ of $(K,v)$, and this is an immediate algebraic extension of $(K,v)$. Thus we
may from now on assume that $(K,v)$ itself is henselian.

\pars
For the proof of part a), we take $(K',v)$ to be an immediate extension of $(K,v)$, maximal with respect
to being contained in $L$. Then $(K',v)$ is henselian since it is an algebraic extension of the henselian
field $(K,v)$. We set $Z_0:=Z\cap L$, where $Z$ is the absolute ramification field of $(K',v)$.
Suppose that $(K',v)$ is not algebraically maximal. Then there exists a nontrivial immediate
algebraic extension $(K'(a)|K',v)$. Since it is purely wild, it is linearly disjoint from the
tame extensions $Z|K'$ and $Z_0|K'$, and it follows that $Z_0(a)|Z_0$ is also linearly disjoint from
$Z|Z_0\,$. By Lemma~\ref{Z}, $Z$ is also the absolute ramification field of $Z_0\,$, so this shows that
$(Z_0(a)|Z_0,v)$ is purely wild. Since by Proposition~\ref{pZ},
$Z_0$ satisfies (K3), all maximal purely wild extensions of $(Z_0,v)$ are isomorphic
over $Z_0$ by Proposition~3.2 of \cite{[KPR]}. As $(L,v)$ is a maximal purely wild extension of $(Z_0,v)$
by Proposition~\ref{pZ}, it follows that $Z_0(a)$ admits an embedding in $L$ over $Z_0\,$. Hence there
is some $a'\in L$ such that $Z_0(a)$ and $Z_0(a')$ are isomorphic over $Z_0\,$. This isomorphism induces
an isomorphism $K'(a)\simeq K'(a')$ over $K'$. Since $(K',v)$ is henselian, this isomorphism preserves the
valuation, so $(K'(a')|K',v)$ is a nontrivial immediate algebraic extension contained in $L$. As this
contradicts the maximality of $(K',v)$, we find that $(K',v)$ must be algebraically maximal.

\pars
In order to prove part b), take $(K',v)$ to be a purely wild extension of $K$, maximal with respect
to being contained in $K$. Again, $(K',v)$ is henselian. By the same proof as before, just replacing
``immediate'' by ``purely wild'', one shows that $(K',v)$ is a maximal purely wild extension of $(K,v)$.
\end{proof}

\sn
{\it Proof of Theorem~\ref{MT}.}
In order to prove the first part of the theorem, we use Proposition~\ref{p2} to replace $L$ by the relative
algebraic closure of $K$ in $L$. Then the assertions follow from Proposition~\ref{prop3}.
\pars
For the proof of the second part of the theorem, assume that $(L,v)$ is maximal, and take an immediate
extension $(K',v)$ of $(K,v)$, maximal
with respect to being contained in $L$. By the first part of our theorem, the relative algebraic closure
of $K'$ in $L$ contains a maximal immediate algebraic extension of $(K',v)$; by the maximality of $K'$, it
must be equal to $K'$. Hence $(K',v)$ does not admit any nontrivial immediate algebraic extensions.

Suppose that $(K',v)$ is not maximal. Then by \cite[Theorem 4]{[Ka]}, $(K',v)$ admits a pseudo Cauchy
sequence without a
limit in $K'$. This must be of transcendental type, because if it were of algebraic type, then by
\cite[Theorem 3]{[Ka]}, there would exist a nontrivial immediate algebraic extension of $(K',v)$,
contradicting its algebraic maximality. Since the pseudo Cauchy sequence we took in $K'$ is also a pseudo
Cauchy sequence in $(L,v)$ and $(L,v)$ is maximal, \cite[Theorem 4]{[Ka]} shows the existence of a limit
$a$ in $L$. It follows from \cite[Theorem 2]{[Ka]} that $(K'(a)|K',v)$ is an immediate transcendental
extension. As it is contained in $L$, this contradicts the maximality of $(K',v)$, We have now proved that
$(K',v)$, is a maximal immediate extension of $(K,v)$.
\qed

\end{document}